\documentclass[11pt]{amsart}
\usepackage[utf8]{inputenc}
\usepackage[T1]{fontenc}
\usepackage{lmodern}
\usepackage{amssymb,amsmath,amsthm,mathtools}
\usepackage{booktabs,array}
\usepackage[margin=1.15in]{geometry}
\usepackage{enumitem}
\usepackage[dvipsnames]{xcolor}
\usepackage[colorlinks=true,linkcolor=blue!50!black,citecolor=green!35!black,urlcolor=blue!50!black]{hyperref}
\hypersetup{pdftitle={Terao's freeness conjecture for arrangements of 15 and 16 lines},pdfauthor={Cristian Anghel, Rares Stan}}

\newtheorem{theorem}{Theorem}[section]
\newtheorem{proposition}[theorem]{Proposition}
\newtheorem{lemma}[theorem]{Lemma}
\newtheorem{corollary}[theorem]{Corollary}
\newtheorem*{maintheorem}{Main Theorem}
\theoremstyle{definition}
\newtheorem{definition}[theorem]{Definition}
\newtheorem{conjecture}[theorem]{Conjecture}
\theoremstyle{remark}
\newtheorem{remark}[theorem]{Remark}

\newcommand{\cA}{\mathcal{A}}
\newcommand{\cB}{\mathcal{B}}
\newcommand{\cL}{\mathcal{L}}
\newcommand{\CC}{\mathbb{C}}
\newcommand{\QQ}{\mathbb{Q}}
\newcommand{\ZZ}{\mathbb{Z}}
\newcommand{\PP}{\mathbb{P}}
\newcommand{\FF}{\mathbb{F}}
\newcommand{\mdr}{\operatorname{mdr}}
\newcommand{\Aut}{\operatorname{Aut}}
\newcommand{\Der}{\operatorname{Der}}

\title[Terao's freeness conjecture for 15 and 16 lines]{Terao's freeness conjecture for\\ arrangements of 15 and 16 lines}

\author{Cristian Anghel}
\address{Simion Stoilow Institute of Mathematics of the Romanian Academy, P.O.\ Box 1-764, RO-014700 Bucharest, Romania}
\email{cristian.anghel@imar.ro}

\author{Rareș Stan}
\address{Simion Stoilow Institute of Mathematics of the Romanian Academy, P.O.\ Box 1-764, RO-014700 Bucharest, Romania}
\email{rstan@imar.ro}


\begin{document}

\begin{abstract}
We prove that Terao's freeness conjecture holds for arrangements of $15$ and of $16$ lines in the complex projective plane. Together with the theorem of Barakat and Kühne for at most $14$ hyperplanes, this shows that freeness of a central arrangement of at most $16$ planes in $\CC^3$ is determined by its intersection lattice. The proof has a theoretical and a computational part. Combining published combinatorial criteria for freeness (due to Abe, Dimca, Dimca--Kühne--Pokora and Yoshinaga) with Bojanowski's Hirzebruch-type inequality and an elementary counting identity, we show that a counterexample would have exponents $(6,8)$ or $(7,7)$ for $15$ lines and $(7,8)$ for $16$ lines, and an intersection lattice subject to strong numerical constraints. We then enumerate all $22\,544\,280$ simple rank $3$ matroids satisfying these constraints, up to isomorphism, and determine the complex realization space of each of them exactly over $\QQ$. Exactly eleven of them are realizable over $\CC$, and on each of these the freeness of the realizations is constant. As a by-product we classify the free arrangements of $15$ and $16$ lines that are not divisionally free, have maximal multiplicity smaller than the smaller nontrivial exponent, and have at least four intersection points on every line.
\end{abstract}

\maketitle

\section{Introduction}

Let $\cA$ be a central arrangement of hyperplanes in $V=\CC^\ell$, let $S=\CC[x_1,\dots,x_\ell]$, and let $Q=\prod_{H\in\cA}\alpha_H$ be a defining polynomial. The module of logarithmic derivations
\[
D(\cA)=\{\theta\in\Der_\CC(S)\ :\ \theta(Q)\in QS\}
\]
is a reflexive graded $S$-module of rank $\ell$, and $\cA$ is called \emph{free} if $D(\cA)$ is a free $S$-module. In that case the degrees $d_1,\dots,d_\ell$ of a homogeneous basis are the \emph{exponents} of $\cA$. Terao's factorization theorem \cite{Ter81} states that the characteristic polynomial then factors as $\chi(\cA,t)=\prod_i(t-d_i)$; in particular the exponents of a free arrangement are determined by its intersection lattice $L(\cA)$. The converse question is the central open problem of the subject.

\begin{conjecture}[Terao, see {\cite[Conj.~4.138]{OT92}}]
Let $\cA$ and $\cB$ be arrangements over the same field with isomorphic intersection lattices. If $\cA$ is free, then $\cB$ is free.
\end{conjecture}

The conjecture is open already for $\ell=3$, i.e.\ for line arrangements in the projective plane, which is the case considered here. Over $\CC$ it has been verified for arrangements of at most $12$ lines by Faenzi and Vallès \cite{FV14}, of $13$ lines by Dimca, Ibadula and Măcinic \cite{DIM19}, and of $14$ lines, over fields of arbitrary characteristic, by Barakat, Behrends, Jefferson, Kühne and Leuner \cite{BBJKL21}. Barakat and Kühne \cite{BK23} proved it for all rank $3$ arrangements with at most $14$ hyperplanes, again over fields of arbitrary characteristic. Their proof enumerates the relevant simple rank $3$ matroids with at most $14$ elements and computes their moduli spaces, in particular their non-free loci. That approach becomes impractical for larger numbers of hyperplanes, since the number of simple rank $3$ matroids grows very fast with the number of elements \cite{MMIB12}.

Our main result is the following.

\begin{maintheorem}
Let $d\in\{15,16\}$, and let $\cA$ and $\cB$ be arrangements of $d$ lines in $\PP^2(\CC)$ with isomorphic intersection lattices. If $\cA$ is free, then $\cB$ is free.
\end{maintheorem}

\begin{corollary}
Terao's conjecture holds for all central arrangements of at most $16$ planes in $\CC^3$.
\end{corollary}

\begin{proof}
Arrangements of rank at most two are free. For rank three, combine the Main Theorem with \cite[Thm.~1.2]{BK23}.
\end{proof}

\subsection*{Strategy}
Instead of examining all matroids with $15$ or $16$ elements, we first use theory to confine a potential counterexample to a very small region. Section~\ref{sec:reduction} combines the following ingredients:
\begin{itemize}
\item results of Dimca \cite{Dim17a} and Dimca--Kühne--Pokora \cite{DKP26}, which force $m\le d_1-1$, where $m$ is the maximal multiplicity of an intersection point and $(d_1,d_2)$, $d_1\le d_2$, are the nontrivial exponents of the free arrangement;
\item Abe's restriction on the number of points on a line of a free arrangement \cite{Abe14} together with his division theorem \cite{Abe16}, which force every line to contain at most $d_1$ intersection points;
\item Yoshinaga's criterion \cite{Yos05}, together with the $3$-transitivity of $\mathrm{PGL}_2$ on $\PP^1$ (compare Wakamiko's explicit formula \cite{Wak07}), which forces every line to contain at least four intersection points;
\item an elementary identity, which we call the \emph{budget identity} (Lemma~\ref{lem:counting}), and Bojanowski's Hirzebruch-type inequality \cite{Lan03,Pok21}.
\end{itemize}
The outcome (Proposition~\ref{prop:reduction}) is that a counterexample with $d\in\{15,16\}$ lines is of type $(d;d_1,d_2)\in\{(15;6,8),(15;7,7),(16;7,8)\}$ and an intersection lattice satisfying five explicit numerical conditions (R1)--(R5).

In Section~\ref{sec:enumeration} we enumerate all simple rank $3$ matroids satisfying (R1)--(R5), up to isomorphism: there are $3$, $6\,091\,145$ and $16\,453\,132$ of them, respectively. In Section~\ref{sec:realization} we compute the complex realization space of each of them exactly over $\QQ$, and we decide freeness on each Galois orbit of realizations by certificates that are verified in exact arithmetic. The results are in Section~\ref{sec:results}. Only $11$ of the $22\,544\,280$ matroids are realizable over $\CC$: one of type $(15;6,8)$, ten of type $(15;7,7)$, and none of type $(16;7,8)$. Each of them is either free for all realizations or non-free for all realizations. This proves the Main Theorem. As a by-product we obtain:

\begin{theorem}\label{thm:classification}
Let $\cA$ be a free arrangement of $d\in\{15,16\}$ lines in $\PP^2(\CC)$ with exponents $(1,d_1,d_2)$, $d_1\le d_2$, and maximal multiplicity $m$. Suppose that $m\le d_1-1$, that every line of $\cA$ contains at least four intersection points, and that $\cA$ is not divisionally free.
\begin{enumerate}[label=\textup{(\alph*)}]
\item If $d=16$, no such arrangement exists.
\item If $d=15$, then $L(\cA)$ is isomorphic to one of the six lattices $L_0,F_1,\dots,F_5$ of Table~\ref{tab:realizable}. They are the lattices of the reflection arrangement $\cA(5,5,3)$ ($L_0$), of Cuntz's one-parameter family \cite[§7]{ACKN16} ($F_1$), of two sub-arrangements of the reflection arrangement of the complex reflection group $G_{26}$ ($F_3$, $F_5$), and of rigid arrangements defined over $\QQ$ ($F_2$) and over $\QQ(\sqrt5)$ ($F_4$).
\end{enumerate}
Conversely, every complex arrangement whose lattice is one of $L_0,F_1,\dots,F_5$ is free and satisfies the hypotheses above.
\end{theorem}

Here $\cA$ is \emph{divisionally free} if $\chi(\cA^H,t)$ divides $\chi(\cA,t)$ for some line $H$ (for $\ell=3$ this agrees with Abe's definition \cite{Abe16}). Divisional freeness, the condition $m\ge d_1$, and the existence of a line with at most three points are three standard situations in which freeness is known to be combinatorial (Section~\ref{sec:reduction}). Theorem~\ref{thm:classification} describes the free arrangements of $15$ and $16$ lines not covered by them. Abe \cite{Abe14} gives further combinatorial criteria. For instance, freeness is combinatorial when $d_1\le5$ \cite[Cor.~5.5]{Abe14} or $d_2\ge2d_1-3$ \cite[Cor.~5.4]{Abe14}. These already exclude the five types that we remove in Table~\ref{tab:reduction} by a different argument. The six lattices of Theorem~\ref{thm:classification} are not covered by these criteria, nor by \cite[Thm.~2.8(2)]{Abe14} (see Remark~\ref{rem:abe28}).

\subsection*{On the computer-assisted part}
The code, the results and a script that reruns the enumeration and the screening are available as ancillary files (Section~\ref{sec:reproducibility}). The programs implement the specification justified in Sections~\ref{sec:reduction}--\ref{sec:realization}.

\section{Preliminaries}\label{sec:prelim}

\subsection{Line arrangements and their lattices}\label{subsec:lattices}
Throughout, $\cA$ is an arrangement of $d$ distinct lines in $\PP^2=\PP^2(\CC)$, or equivalently a central arrangement of $d$ planes in $\CC^3$. A pencil (all lines through one point) is always free, so we may and do assume that the lines are not all concurrent. For an intersection point $P$ let $m_P$ be the number of lines through $P$. Let $t_r$ be the number of points of multiplicity $r$, and let $m=\max_P m_P$. For a line $H$ let $r_H=|\cA^H|$ be the number of intersection points on $H$, and put
\[
e(H)=\sum_{P\in H}(m_P-2).
\]
Since every other line meets $H$ in exactly one point, $\sum_{P\in H}(m_P-1)=d-1$, and hence
\begin{equation}\label{eq:eH}
e(H)=d-1-r_H.
\end{equation}
The intersection lattice $L(\cA)$ is determined by the simple rank $3$ matroid on the set of lines whose rank $2$ flats are the sets of lines through the intersection points. We encode it by the family of its \emph{blocks}, the sets of lines through the points with $m_P\ge3$; any two blocks share at most one line. We identify $L(\cA)$ with this matroid and use the words lattice and matroid interchangeably; in the language of incidence geometry, such a matroid is a \emph{linear space}. All quantities above depend only on the lattice. The characteristic polynomial is $\chi(\cA,t)=(t-1)\chi_0(\cA,t)$ with
\begin{equation}\label{eq:chi}
\chi_0(\cA,t)=t^2-(d-1)t+b,\qquad b=\sum_P(m_P-1)-(d-1).
\end{equation}

\subsection{Freeness and exponents}
If $\cA$ is free, its exponents are $(1,d_1,d_2)$ with $d_1+d_2=d-1$ and $d_1d_2=b$ \cite{Ter81}. We always order them so that $d_1\le d_2$, and we say that $L(\cA)$ is of \emph{type} $(d;d_1,d_2)$ whenever $\chi_0(\cA,t)=(t-d_1)(t-d_2)$ with integers $d_1\le d_2$.

Let $f=Q(\cA)$, and let $\mdr(f)$ be the minimal degree of a nonzero syzygy $af_x+bf_y+cf_z=0$. These Jacobian syzygies form the module $D_0(\cA)$, and $D(\cA)=S\theta_E\oplus D_0(\cA)$, where $\theta_E$ is the Euler derivation. Hence a free arrangement has $\mdr(f)=d_1$. Let $\tau(\cA)=\sum_P(m_P-1)^2$ be the global Tjurina number of $f$, and put $\tau_{\max}(d,r)=(d-1)^2-r(d-1-r)$.

\begin{theorem}[{\cite[Cor.~1.2]{Dim17b}}, based on du Plessis--Wall \cite{dPW99}]\label{thm:dPW}
If $r=\mdr(f)<d/2$, then $\cA$ is free if and only if $\tau(\cA)=\tau_{\max}(d,r)$.
\end{theorem}

\subsection{Ziegler's multirestriction and Yoshinaga's criterion}\label{subsec:ziegler}
For a line $H\in\cA$, the \emph{Ziegler multirestriction} \cite{Zie89} is the multiarrangement $(\cA^H,\mu^H)$ on $H\cong\PP^1$ consisting of the $r_H$ points of $H$, where $P$ has multiplicity $\mu^H(P)=m_P-1$. It is free (being of rank $2$), with exponents $(e_1^H,e_2^H)$, $e_1^H\le e_2^H$, $e_1^H+e_2^H=\sum_{P\in H}(m_P-1)=d-1$. Moreover $e_1^H$ is the minimal degree of a nonzero element of $D(\cA^H,\mu^H)$.

\begin{theorem}[Yoshinaga {\cite[Thm.~3.2]{Yos05}}; see also {\cite[Thm.~2.3]{Abe14}}]\label{thm:yoshinaga}
For every line $H$ of an arrangement $\cA$ of lines in $\PP^2(\CC)$, $b\ge e_1^He_2^H$, and $\cA$ is free if and only if $b=e_1^He_2^H$.
\end{theorem}

\begin{lemma}\label{lem:freecrit}
Let $L(\cA)$ be of type $(d;d_1,d_2)$ and let $H\in\cA$. Then $\cA$ is free if and only if $e_1^H\ge d_1$, i.e.\ if and only if $D(\cA^H,\mu^H)$ has no nonzero element of degree $d_1-1$.
\end{lemma}

\begin{proof}
If $\cA$ is free, then $e_1^He_2^H=d_1d_2$ and $e_1^H+e_2^H=d_1+d_2$, so $e_1^H=d_1$. Conversely, the function $x\mapsto x(d-1-x)$ is increasing for $x\le(d-1)/2$, so $d_1\le e_1^H\le(d-1)/2$ implies $e_1^He_2^H\ge d_1d_2=b$. Theorem~\ref{thm:yoshinaga} then gives equality, hence freeness.
\end{proof}

Concretely, let $H\cong\PP^1$ have homogeneous coordinates $(X:Y)$ and write $P=(s_P:t_P)$. A derivation $\theta=g_1\partial_X+g_2\partial_Y$, with $g_1,g_2$ homogeneous of degree $k$, lies in $D(\cA^H,\mu^H)$ if and only if $(t_PX-s_PY)^{\mu^H(P)}$ divides $h_P:=t_Pg_1-s_Pg_2$ for every $P\in H$. If $t_P\neq0$, this divisibility is equivalent to the vanishing at $(s_P,t_P)$ of the Hasse derivatives $\partial_X^{(j)}h_P$, $0\le j<\mu^H(P)$; if $t_P=0$ one uses $\partial_Y^{(j)}$ instead. This is valid in every characteristic. We obtain a system of $d-1$ linear equations in the $2(k+1)$ coefficients of $g_1,g_2$, and we denote its matrix by $M_{H,k}$. If $X$-derivatives are used at every point, the resulting system is implied by the exact one, so its solution space \emph{contains} $D(\cA^H,\mu^H)_k$, with equality when all $t_P\ne0$.

\section{Reduction to three lattice types}\label{sec:reduction}

\begin{definition}
A lattice $L$ is a \emph{Terao counterexample} if there are complex line arrangements $\cA$, $\cB$ with $L(\cA)\cong L(\cB)\cong L$, $\cA$ free and $\cB$ not free.
\end{definition}

In this section $L$ is a Terao counterexample with $d$ lines, realized by $\cA$ (free) and $\cB$ (not free), and $(d_1,d_2)$, $d_1\le d_2$, are the exponents of $\cA$. We put
\[
X=\sum_P\binom{m_P-1}{2},\qquad W=\sum_P\binom{m_P-2}{2}=t_4+3t_5+6t_6+\cdots.
\]

\begin{lemma}[counting]\label{lem:counting}
For every lattice of type $(d;d_1,d_2)$:
\begin{enumerate}[label=\textup{(\alph*)}]
\item $\sum_P\binom{m_P}{2}=\binom d2$ and $\sum_P(m_P-1)=d_1d_2+d-1$;
\item $X=\binom{d-1}{2}-d_1d_2$ and $\tau=(d-1)^2-d_1d_2=\tau_{\max}(d,d_1)$;
\item $\sum_He(H)=3X-W$; in particular, with $\beta:=3X-d\,d_2$,
\begin{equation}\label{eq:budget}
W+\sum_H\bigl(e(H)-d_2\bigr)=\beta\qquad\text{(budget identity).}
\end{equation}
\end{enumerate}
Conversely, for every simple rank $3$ matroid on $d$ elements, $X=\binom d2-\sum_P(m_P-1)=\binom d2-(b+d-1)$ by \eqref{eq:chi}. So $X$ determines $\chi_0$, and every lattice with $X=\binom{d-1}2-d_1d_2$ is of type $(d;d_1,d_2)$.
\end{lemma}

\begin{proof}
(a) The first identity counts pairs of lines; the second is \eqref{eq:chi} with $b=d_1d_2$. (b) $X=\sum_P\binom{m_P}{2}-\sum_P(m_P-1)$. For $\tau$, note that $(m-1)^2=2\binom{m-1}{2}+(m-1)$, so $\tau=2X+d_1d_2+d-1=(d-1)^2-d_1d_2$. (c) $\sum_He(H)=\sum_Pm_P(m_P-2)$, and $m(m-2)=3\binom{m-1}{2}-\binom{m-2}{2}$ for all $m$.
\end{proof}

\begin{lemma}\label{lem:dimca}
$m\le d_1-1$.
\end{lemma}

\begin{proof}
This is the first assertion of \cite[Cor.~1.4]{Dim17a}: Terao's conjecture holds for a free arrangement $\cA$ with $m\ge d_1$. For completeness we include a short argument; in the case $d_1\in\{m-1,m\}$ it uses \cite[Props.~2.1--2.2]{DKP26} instead of the original argument of \cite{Dim17a}. We have $\mdr(f_\cA)=d_1$. By Dimca's theorem \cite[Thm.~1.2]{Dim17a}, applied to a point of maximal multiplicity $m$, either $d_1=d-m$ or $m-1\le d_1\le d-m-1$.

\emph{Case $d_1=d-m$.} This case is covered by \cite[Cor.~1.4(1)]{Dim17a}; we give the argument. Then $m=d-d_1\ge d_2+1>d_1$. Apply \cite[Thm.~1.2]{Dim17a} to $\cB$, which has the same $m$. Either $\mdr(f_\cB)=d-m=d_1$, or $m-1\le\mdr(f_\cB)\le d-m-1=d_1-1$. The second alternative would give $m\le d_1$, which is impossible. Hence $\mdr(f_\cB)=d_1\le(d-1)/2<d/2$, and $\tau(\cB)=\tau(\cA)=\tau_{\max}(d,d_1)$ by Lemma~\ref{lem:counting}(b). By Theorem~\ref{thm:dPW}, $\cB$ is free, a contradiction.

\emph{Case $d_1\in\{m-1,m\}$.} By \cite[Prop.~2.1]{DKP26}, a line arrangement in $\PP^2(\CC)$ with $d$ lines, maximal multiplicity $m$ and $\tau=\tau_{\max}(d,m-1)$ is free, with exponents $(m-1,d-m)$. By \cite[Prop.~2.2]{DKP26}, if instead $\tau=\tau_{\max}(d,m)$, it is free with exponents $(m,d-m-1)$. Now $m(\cB)=m$, and $\tau(\cB)=\tau(\cA)=\tau_{\max}(d,d_1)$ by Lemma~\ref{lem:counting}(b), which is $\tau_{\max}(d,m-1)$ or $\tau_{\max}(d,m)$. Hence $\cB$ is free, a contradiction.

Hence $d_1\ge m+1$.
\end{proof}

\begin{lemma}\label{lem:abe}
$r_H\le d_1$, i.e.\ $e(H)\ge d_2$, for every line $H$.
\end{lemma}

\begin{proof}
Fix $H$, and choose $H_0\in\cA$ with $H_0\ne H$. The deconing of $\cA$ with respect to $H_0$ is the affine arrangement $\mathbf{d}\cA$ of the $d-1$ lines $H'\setminus H_0$, $H'\ne H_0$, in $\PP^2\setminus H_0\cong\CC^2$. Its cone is $\cA$, so it is free in the sense of \cite{Abe14}, and $\chi(\mathbf{d}\cA,t)=(t-d_1)(t-d_2)$. The affine line $H\setminus H_0$ contains exactly the intersection points of $\cA$ on $H$ other than $H\cap H_0$, i.e.\ $r_H-1$ of them. By \cite[Cor.~1.2(1)]{Abe14}, $r_H-1\le d_1$ or $r_H-1=d_2$.

Suppose that $r_H\in\{d_1+1,d_2+1\}$. Then $\chi(\cB^H,t)=(t-1)(t-r_H+1)$ divides $\chi(\cB,t)=(t-1)(t-d_1)(t-d_2)$, and $\cB^H$ is free since it has rank $2$. By Abe's division theorem \cite[Thm.~1.1]{Abe16}, $\cB$ is free, a contradiction. Hence $r_H\le d_1$. The reformulation in terms of $e(H)$ is \eqref{eq:eH}.
\end{proof}

\begin{lemma}\label{lem:wakamiko}
$r_H\ge4$, i.e.\ $e(H)\le d-5$, for every line $H$.
\end{lemma}

\begin{proof}
Suppose that $r_H\le3$. The Ziegler multirestrictions of $\cA$ and $\cB$ on $H$ consist of at most three points with the same multiplicities. Since $\mathrm{PGL}_2(\CC)$ acts $3$-transitively on $\PP^1$, they are isomorphic multiarrangements, so $e_1^H$ is the same for $\cA$ and $\cB$. (For three points the exponents are given explicitly by Wakamiko's formula \cite{Wak07}.) Since $\cA$ is free, $e_1^H=d_1$ by Lemma~\ref{lem:freecrit}, and then $\cB$ is free by the same lemma, a contradiction.
\end{proof}

\begin{remark}\label{rem:abe28}
All arrangements considered here are \emph{balanced}: every multiplicity $\mu^H(P)=m_P-1$ is at most half of $d-1$, because $m\le d_1-1\le (d-3)/2$. For balanced arrangements, \cite[Thm.~2.8]{Abe14} gives $e_2^H-e_1^H\le r_H-2$, and freeness is combinatorial as soon as $d_2-d_1\in\{r_H-2,r_H-3\}$ for some line $H$. This would sharpen Lemma~\ref{lem:wakamiko} to $r_H\ge d_2-d_1+4$, which we do not need. For the lattices of Theorem~\ref{thm:classification} we have $d_2-d_1\in\{0,2\}$ and $r_H\ge6$, so this criterion does not apply to them.
\end{remark}

\begin{lemma}[Bojanowski's inequality {\cite{Boj03}}; see {\cite[Thm.~19]{Pok21}}, {\cite[Rem.~2.5]{Pok17}}]\label{lem:boj}
Let $\cA$ be an arrangement of $d\ge6$ lines in $\PP^2(\CC)$ with $t_r=0$ for $r>2d/3$. Then
\[
t_2+\tfrac34t_3\ \ge\ d+\sum_{r\ge5}\Bigl(\frac{r^2}4-r\Bigr)t_r.
\]
If $L(\cA)$ is of type $(d;d_1,d_2)$, this inequality is equivalent to
\begin{equation}\label{eq:wmin}
3W\ \ge\ 5X-4d_1d_2+4.
\end{equation}
\end{lemma}

Bojanowski's inequality follows from an inequality of Langer for line arrangements, \cite[Prop.~11.3.1]{Lan03}, which Langer derives from his orbifold Miyaoka--Yau inequality.

\begin{proof}[Proof of the equivalence with \eqref{eq:wmin}]
For every integer $k$,
\[
\frac54\binom{k-1}{2}-(k-1)-\Bigl(\frac{k^2}4-k\Bigr)=\frac34\binom{k-2}{2},
\]
as both sides equal $\frac38k^2-\frac{15}8k+\frac94$. Since $k^2/4-k$ equals $-1$, $-\frac34$ and $0$ for $k=2,3,4$, summing over all points and using Lemma~\ref{lem:counting}(a) gives
\[
\frac34W=\frac54X-(d_1d_2+d-1)+\Bigl(t_2+\frac34t_3-\sum_{r\ge5}\Bigl(\frac{r^2}4-r\Bigr)t_r\Bigr).
\]
Hence Bojanowski's inequality is equivalent to $\frac34W\ge\frac54X-d_1d_2+1$.
\end{proof}

\begin{proposition}\label{prop:reduction}
Let $L$ be a Terao counterexample with $d\in\{15,16\}$ lines, of type $(d;d_1,d_2)$. Then
\[
(d;d_1,d_2)\in\{(15;6,8),\ (15;7,7),\ (16;7,8)\},
\]
and $L$ satisfies:
\begin{enumerate}[label=\textup{(R\arabic*)},leftmargin=3.2em]
\item $3\le m_P\le d_1-1$ for every point with $m_P\ge3$;
\item $d_2\le e(H)\le d-5$ for every line $H$;
\item $X=\binom{d-1}{2}-d_1d_2$;
\item $W\ge W_{\min}:=\bigl\lceil(5X-4d_1d_2+4)/3\bigr\rceil$;
\item $W+\sum_H(e(H)-d_2)=\beta=3X-d\,d_2$.
\end{enumerate}
\end{proposition}

\begin{proof}
(R1)--(R5) follow from Lemmas~\ref{lem:dimca}, \ref{lem:abe}, \ref{lem:wakamiko}, \ref{lem:counting} and \ref{lem:boj}. Lemma~\ref{lem:boj} applies because $d\ge 6$ and $m\le d_1-1\le(d-1)/2-1<2d/3$. Since $X=\binom{d-1}2-d_1d_2\ge\binom{d-1}2-\frac{(d-1)^2}4>0$, some point has $m_P\ge3$, so $m\ge3$ and Lemma~\ref{lem:dimca} gives $d_1\ge4$. By Lemma~\ref{lem:counting}(c), (R5) holds for every lattice satisfying (R3); we list it because it is used for pruning. By (R2) every summand in (R5) is nonnegative, so $W_{\min}\le W\le\beta$. Table~\ref{tab:reduction} lists all types with $4\le d_1\le(d-1)/2$; all but the three stated ones have $W_{\min}>\beta$.
\end{proof}

\begin{table}[ht]
\centering
\begin{tabular}{cccccccc}
\toprule
$d$ & $(d_1,d_2)$ & $X$ & $\beta=3X-d\,d_2$ & $5X-4d_1d_2+4$ & $W_{\min}$ & $m\le$ & type\\
\midrule
15 & $(4,10)$ & 51 & 3 & 99 & 33 & 3 & excluded\\
15 & $(5,9)$ & 46 & 3 & 54 & 18 & 4 & excluded\\
15 & $(6,8)$ & 43 & 9 & 27 & 9 & 5 & remains\\
15 & $(7,7)$ & 42 & 21 & 18 & 6 & 6 & remains\\
16 & $(4,11)$ & 61 & 7 & 133 & 45 & 3 & excluded\\
16 & $(5,10)$ & 55 & 5 & 79 & 27 & 4 & excluded\\
16 & $(6,9)$ & 51 & 9 & 43 & 15 & 5 & excluded\\
16 & $(7,8)$ & 49 & 19 & 25 & 9 & 6 & remains\\
\bottomrule
\end{tabular}
\caption{The numerical reduction (Proposition~\ref{prop:reduction}).}
\label{tab:reduction}
\end{table}

\begin{remark}
Each line has two degrees of freedom, a projective frame fixes four lines, and a point of multiplicity $m_P$ imposes $m_P-2$ conditions. The expected dimension of the space of realizations modulo $\mathrm{PGL}_3$ is therefore
\[
2(d-4)-\sum_P(m_P-2)=2d-8-(X-W),
\]
using $m-2=\binom{m-1}{2}-\binom{m-2}{2}$. In the three remaining types it equals $-12$ (type $(15;6,8)$, where (R4) and (R5) force $W=9$), $W-20\le1$, and $W-25\le-6$, respectively. This is only a parameter count, not a dimension bound: the incidence conditions need not be independent. It suggests that realizable lattices may be rare, which is what we find.
\end{remark}

\begin{remark}
In type $(15;6,8)$ we have $W_{\min}=\beta=9$. By (R5) this forces $W=9$ and $e(H)=8$, i.e.\ $r_H=6$, for every line.
\end{remark}

\section{Enumeration of the residual lattices}\label{sec:enumeration}

For a type $(d;d_1,d_2)$ from Proposition~\ref{prop:reduction}, let $\cL(d;d_1,d_2)$ be the set of isomorphism classes of simple rank $3$ matroids on $d$ elements satisfying (R1)--(R5). By Proposition~\ref{prop:reduction} and Lemma~\ref{lem:counting}, every Terao counterexample with $15$ or $16$ lines lies in one of the three sets $\cL(15;6,8)$, $\cL(15;7,7)$, $\cL(16;7,8)$. In this section we enumerate these sets.

\subsection{Encoding}
A matroid is given by its family of blocks (Section~\ref{subsec:lattices}). Two matroids are isomorphic if and only if their bipartite point--block incidence graphs are isomorphic by an isomorphism preserving the bipartition. We use the canonical labelling of nauty \cite{MP14} (program \texttt{labelg}, with the points forming the first colour class) and store one canonical \texttt{graph6} string per class.

\subsection{Stage 1: heavy configurations}
For a matroid $L$, let $\mathcal H(L)$ be its set of blocks of size at least $4$, the \emph{heavy configuration} of $L$. Triple points contribute $0$ to $W$, so $W(L)=W(\mathcal H(L))$. For an element $v$ of the matroid (a line of the arrangement; in this section we call the elements \emph{points}, as in the theory of linear spaces) put
\[
h(v)=\sum_{B\in\mathcal H,\,v\in B}(|B|-2),\qquad X_{\mathcal H}=\sum_{B\in\mathcal H}\binom{|B|-1}{2}.
\]
Then $e(v)=h(v)+\tau_3(v)$, where $\tau_3(v)$ is the number of triple points through $v$, and $t_3=X-X_{\mathcal H}$. Put $h_{\max}=\min(d-5,\ d_2+\beta-W_{\min})$.

\begin{lemma}\label{lem:heavy}
If $L$ satisfies (R1)--(R5), then $\mathcal H=\mathcal H(L)$ satisfies:
\begin{enumerate}[label=\textup{(H\arabic*)},leftmargin=3.2em]
\item its blocks have sizes in $[4,d_1-1]$ and pairwise share at most one element;
\item $h(v)\le h_{\max}$ for all $v$, and $X_{\mathcal H}\le X$;
\item $W(\mathcal H)+\sum_v(h(v)-d_2)^+\le\beta$, and $\sum_v(h(v)-d_2)^+\le\beta-W_{\min}$;
\item $W(\mathcal H)\ge W_{\min}$;
\item for each $v$, let $u(v)$ be the number of $w\ne v$ not in a common block of $\mathcal H$ with $v$, and put $c(v)=\min(\lfloor u(v)/2\rfloor,\,h_{\max}-h(v))$. Then $(d_2-h(v))^+\le c(v)$ for every $v$, and
\[
\sum_v(d_2-h(v))^+\ \le\ 3(X-X_{\mathcal H})\ \le\ \sum_vc(v).
\]
\end{enumerate}
Conditions \textup{(H1)--(H3)} are inherited by every subset of $\mathcal H$.
\end{lemma}

\begin{proof}
(H1) is (R1). For (H2): $h(v)\le e(v)\le d-5$, by (R5) $e(v)-d_2\le\beta-W\le\beta-W_{\min}$, and $X-X_{\mathcal H}=t_3\ge0$. For (H3): $e(v)\ge\max(h(v),d_2)$, so $e(v)-d_2\ge(h(v)-d_2)^+$; now use (R5) and (R4). (H4) is (R4). For (H5): the triple points through $v$ use disjoint pairs among the $u(v)$ elements not yet collinear with $v$. Moreover $\tau_3(v)=e(v)-h(v)$ lies between $(d_2-h(v))^+$ and $h_{\max}-h(v)$, and $\sum_v\tau_3(v)=3t_3$. Heredity holds because $h$, $X_{\mathcal H}$ and $W$ can only decrease when blocks are removed.
\end{proof}

We generate all heavy configurations satisfying (H1)--(H3), one per isomorphism class, by McKay's canonical augmentation \cite{McK98}, adding one block at a time, and output those satisfying (H4) and (H5).
\begin{itemize}
\item The candidate blocks are the sets of size $4,\dots,d_1-1$ containing no pair already covered, and containing no $v$ with $h(v)+2>h_{\max}$.
\item The canonical deletion of a configuration is the block whose vertex comes last in the canonical labelling of the incidence graph. Blocks are coloured by size, after the points, so the last vertex is always a block.
\item A child is accepted if and only if the new block lies in the orbit of the canonical deletion under the automorphism group of the child. Isomorphic accepted siblings are identified by comparing canonical graphs.
\end{itemize}
Every child is tested again against (H1)--(H3), since a candidate block of size at least $5$ may overshoot $h_{\max}$. McKay's scheme \cite{McK98} applies because:
\begin{itemize}
\item (H1)--(H3) are hereditary, so deleting the canonical block of a valid configuration gives a valid configuration;
\item the candidate filter admits the canonical block of every valid child, since $|B|-2\ge2$;
\item the acceptance test is isomorphism invariant, because the ordered colour partition is fixed.
\end{itemize}
Hence this produces every isomorphism class of configurations satisfying (H1)--(H3), each at least once; the final deduplication with \texttt{labelg} removes any remaining repetition.

\subsection{Stage 2: completion by triple points}
Let $\mathcal H$ be a representative heavy configuration. A \emph{completion} is a set $T$ of $3$-subsets such that $\mathcal H\cup T$ is a linear space with heavy configuration $\mathcal H$ satisfying (R1)--(R5). We enumerate the completions by backtracking over the points $v=0,1,\dots,d-1$, ordered by decreasing $h(v)$. At step $v$ we choose all triples whose smallest element is $v$, so that $e(v)$ takes its final value, and we prune with (R2), (R5) and capacity bounds as in (H5).

We break symmetry as follows. At step $v$, the points $w>v$ that lie in exactly the same existing blocks form classes, and new triples through $v$ use only initial segments of the currently unused members of each class. Permuting the remaining unused members of such a class fixes the pre-existing blocks, all processed points, and the triples already chosen at this step. Induction over the triple choices therefore shows that every completion is equivalent to one respecting this rule.

Every $L$ satisfying (R1)--(R5) is isomorphic to a completion of the representative of $\mathcal H(L)$. Therefore the completions of all representatives, after deduplication by canonical form, are exactly the elements of $\cL(d;d_1,d_2)$.

\subsection{Results of the enumeration}
\begin{center}
\begin{tabular}{lrrr}
\toprule
type & heavy configurations & completions found & classes $|\cL|$\\
\midrule
$(15;6,8)$ & $1\,952$ & $1\,346$ & $3$\\
$(15;7,7)$ & $96\,725$ & $10\,016\,867$ & $6\,091\,145$\\
$(16;7,8)$ & $8\,324\,811$ & $17\,590\,573$ & $16\,453\,132$\\
\bottomrule
\end{tabular}
\end{center}
For $(16;7,8)$, the canonical augmentation visited $8\,885\,007$ nodes (about $9$ minutes on one core), and the completion visited $2.97\times10^{10}$ nodes (about $25$ minutes).

\section{Exact computation of realization spaces}\label{sec:realization}

Let $L$ be a simple rank $3$ matroid on $\{1,\dots,d\}$. A \emph{realization} of $L$ over a field $K$ (of any characteristic) is an arrangement $(\ell_1,\dots,\ell_d)$ of lines in $\PP^2(K)$, labelled by the elements of $L$, whose labelled lattice, defined exactly as over $\CC$, is $L$.

\subsection{Normal form and construction sequence}
Suppose that $L$ admits a \emph{frame}: four elements, no three of them in a common block. Every residual matroid used below is explicitly checked to admit such a frame. Since $\mathrm{PGL}_3$ acts simply transitively on ordered $4$-tuples of lines in general position, every realization is projectively equivalent to a unique one in which the frame lines are $x=0$, $y=0$, $z=0$, $x+y+z=0$. Order the remaining elements. When an element $\ell$ is placed, call a block \emph{determined} if it contains $\ell$ and at least two placed elements; the corresponding point is then already constructed. There are three kinds of steps:
\begin{itemize}
\item \emph{det}: $\ell$ lies in at least two determined blocks, so it is the line through two constructed points;
\item \emph{pencil}: $\ell$ lies in exactly one determined block, through the point $P=\ell_a\cap\ell_b$, so $\ell=\ell_a+t\,\ell_b$ for a parameter $t$;
\item \emph{free}: $\ell$ lies in no determined block, so it depends on two parameters.
\end{itemize}
In a pencil step, the only line through $P$ not of the form $\ell_a+t\ell_b$ is $\ell_b$, which is already one of the lines of $L$. So no realization is lost. The \emph{cost} of the sequence is the number of parameters. We choose frames and orders greedily so as to minimize the cost; correctness does not depend on this choice, and the program checks explicitly that the chosen frame is in general position. With our greedy choice, all $22\,544\,280$ matroids except one received a sequence of cost at most $1$.

\subsection{One parameter}
For cost $\le1$, every line is represented by a vector $\ell_i(t)\in\ZZ[t]^3$, built by cross products and pencils. Each vector is made primitive, and every common factor removed is recorded as a \emph{forbidden} polynomial; at its roots two distinct points or lines of the construction would coincide. Let $G\in\ZZ[t]$ be the gcd of the \emph{required} determinants $\det(\ell_{b_0},\ell_{b_1},\ell_x)$, for all blocks $B\ni b_0,b_1,x$. Let $N$ consist of the forbidden polynomials and of the determinants $\det(\ell_i,\ell_j,\ell_k)$ of all triples not contained in a block.

\begin{proposition}\label{prop:onepar}
Suppose the sequence has cost $1$. Every complex realization of $L$ is projectively equivalent to exactly one arrangement $(\ell_i(t_0))_i$ with $t_0\in\CC$. This arrangement is a realization if and only if $G(t_0)=0$ and $n(t_0)\ne0$ for all $n\in N$. Consequently:
\begin{enumerate}[label=\textup{(\alph*)}]
\item if some $\ell_i$ is identically $0$, or $G$ is a nonzero constant, or every irreducible factor of $G$ divides an element of $N$, or $G=0$ and some element of $N$ is identically $0$, then $L$ is not realizable over $\CC$;
\item if $G\ne0$, the realizations correspond to the roots of the irreducible factors $g$ of $G$ that divide no element of $N$, and the realizations coming from one factor $g$ form a single orbit under $\Aut(\CC/\QQ)$;
\item if $G=0$, the realizations form the Zariski open subset $\{t_0:\ n(t_0)\ne0\ \forall n\in N\}$ of $\mathbb A^1$.
\end{enumerate}
\end{proposition}

\begin{proof}
The construction reproduces any normalized realization step by step: det steps are forced, and pencil steps were discussed above. If a constructed vector vanishes at $t_0$, then two elements that are distinct in $L$ give the same point or line. The conditions $G(t_0)=0$ and $n(t_0)\ne0$ say exactly that the blocks of $L$ are collinear and that no other triple is. For an irreducible $g\in\QQ[t]$ and $n\in\QQ[t]$, $n$ vanishes at one root of $g$ if and only if $g\mid n$, if and only if $n$ vanishes at all roots of $g$; and the roots of $g$ are permuted transitively by $\Aut(\CC/\QQ)$.
\end{proof}

By the same argument, if the sequence has cost $0$, the vectors $\ell_i$ are integral, $G$ is an integer, and $L$ has at most one realization up to projective equivalence. It exists, and is then defined over $\QQ$, if and only if $G=0$ and no element of $N$ vanishes.

\subsection{Several parameters}\label{subsec:several}
For the single matroid of cost $2$, the construction has two pencil steps and no free step. The same cross-product construction therefore gives vectors in $\ZZ[t_1,t_2]^3$ and covers every normalized realization. Required determinants must vanish and nonblock determinants must be nonzero, exactly as in the one-parameter case.

This class has the following elementary obstruction, verified directly in the ancillary module \texttt{exact\_multi.py} and in the script \path{identifications/certificates.py}. It is the class with index $13095$ (counted from $0$) in the list of type $(15;7,7)$; its elements are numbered from $0$. With the frame $(3,4,7,8)$ and parameters $t_1,t_2$, the construction gives
\[
q=\det(\ell_1,\ell_2,\ell_{14})=-4t_1^2+4t_1-1,\qquad f=\det(\ell_0,\ell_1,\ell_7)=2t_1-1.
\]
The triple $\{1,2,14\}$ lies in a block and the triple $\{0,1,7\}$ does not, so a realization would satisfy $q=0$ and $f\ne0$. This is impossible, since $q+f^2=0$ identically.

\subsection{Deciding freeness}
Freeness is invariant under $\Aut(\CC/\QQ)$ acting on the coefficients, so by Proposition~\ref{prop:onepar}(b) it is constant on each orbit of realizations. By Lemma~\ref{lem:freecrit}, deciding freeness of a realization of type $(d;d_1,d_2)$ amounts to deciding whether the matrix $M_{H,d_1-1}$ has trivial kernel, for one line $H$. Coordinates on $H$ are given by two fixed integral linear forms whose common zero does not lie on $H$.

\begin{proposition}[certificates]\label{prop:cert}
Let $g\in\ZZ[t]$ be a primitive irreducible factor as in Proposition~\ref{prop:onepar}\textup{(b)}, let $K=\QQ[t]/(g)$, and let $\cA_K$ be the corresponding realization.
\begin{enumerate}[label=\textup{(\alph*)}]
\item \emph{(free)} Suppose $p$ is a prime not dividing the leading coefficient of $g$, $r\in\FF_p$ is a root of $g\bmod p$, and the reduced arrangement $\cA_p=(\ell_i(r))_i$ over $\FF_p$ has the same lattice. If $D(\cA_p^H,\mu^H)$ has no nonzero element of degree $d_1-1$ (computed with the exact system of Section~\ref{subsec:ziegler} over $\FF_p$), then $\cA_K$ is free.
\item \emph{(non-free)} Suppose all points $P$ of $H$ have $t_P\ne0$ in $K$. If $M_{H,d_1-1}$ has a nonzero kernel over $K$, computed by exact Gaussian elimination in $K$, then $\cA_K$ is not free.
\end{enumerate}
\end{proposition}

\begin{proof}
We apply Lemma~\ref{lem:freecrit} to $\cA_K$ through an embedding $K\hookrightarrow\CC$. (a) Let $\Lambda$ be the image of $\ZZ[1/\mathrm{lc}(g)][t]$ in $K$; it maps to $\FF_p$ by $t\mapsto r$. Choose coordinate forms on $H$ with integral coefficients whose common zero avoids $H$ also modulo $p$. The coordinates $(s_P,t_P)$ of the points of $H$ then lie in $\Lambda$ and reduce to the coordinates of the points of $\cA_p$ on $H$, which are distinct because the lattices agree. At each point use the $X$-direction if $t_P\not\equiv0$ and the $Y$-direction otherwise. The resulting matrix over $\Lambda$ is exact both over $K$ and, after reduction, over $\FF_p$. Its rank over $\FF_p$ is at most its rank over $K$, since a nonzero maximal minor modulo $p$ is the image of a nonzero minor. Hence a trivial kernel modulo $p$ gives a trivial kernel over $K$, so $e_1^H\ge d_1$ and $\cA_K$ is free. (b) Under the hypothesis the system is exact, so $e_1^H\le d_1-1$ and $\cA_K$ is not free.
\end{proof}

For a family ($G=0$), let $D_H\in\ZZ[t]$ be the determinant of $M_{H,d_1-1}(t)$ when the matrix is square. Otherwise let $D_H$ be the gcd of the determinants of a few random integral square combinations $R\cdot M_{H,d_1-1}(t)$. If a member $\cA_{t_0}$ is not free, then the rank of $M_{H,d_1-1}(t_0)$ drops and $D_H(t_0)=0$. This holds even where the $X$-direction encoding is not exact, because that encoding only enlarges the kernel (Section~\ref{subsec:ziegler}). We take three lines $H$ such that the common zero of the two coordinate forms lies on $H(t_0)$ for no $t_0$ in the realization space, so that the forms restrict to coordinates on every $H(t_0)$, and let $D=\gcd_HD_H$. (For $F_1$ and $F_2$, the only lattices to which this criterion is applied, the three lines and the forms do not depend on $t$. If a point of $H$ with $\mu^H(P)\ge2$ had $t_P\equiv0$, then $D_H$ would vanish identically and the criterion would simply fail.) If $D\ne0$ and every irreducible factor of $D$ divides an element of $N$, then every realization in the family is free. If $D$ is a nonzero constant, all members are free.

In all certificates of type (b) that occur, we checked that every $t_P$ is nonzero in $K$.

\section{Results}\label{sec:results}

Applied to all classes of Section~\ref{sec:enumeration} (we call this the \emph{screening}), the computation of Section~\ref{sec:realization} gives the following verdicts.

\begin{center}\small
\begin{tabular}{p{6.4cm}rrr}
\toprule
& $(15;6,8)$ & $(15;7,7)$ & $(16;7,8)$\\
\midrule
classes & $3$ & $6\,091\,145$ & $16\,453\,132$\\
not realizable: $G$ constant, or each factor of $G$ divides an element of $N$ & $2$ & $6\,021\,701$ & $16\,175\,061$\\
not realizable: a constructed vector vanishes identically & $0$ & $69\,433$ & $278\,071$\\
not realizable: elementary certificate (cost $2$) & $0$ & $1$ & $0$\\
realizable: finitely many (one Galois orbit) / one family & $1$ / $0$ & $9$ / $1$ & $0$ / $0$\\
\bottomrule
\end{tabular}
\end{center}

The remaining case of Proposition~\ref{prop:onepar}(a), $G=0$ with an element of $N$ vanishing identically, does not occur. Each realizable lattice has a single Galois orbit of realizations, or a single family. The orbit of $F_2$ has one element, since its construction has cost $0$. (The program treats this cost-$0$ case like a family with constant lines, so the ancillary file \texttt{results/verdict\_15\_77.txt} lists $8$ classes with status \texttt{points} and $2$ with status \texttt{family}; in the table above $F_2$ is counted among the lattices with finitely many realizations.) All eleven are listed in Table~\ref{tab:realizable}. Their blocks are listed in Appendix~\ref{app:blocks}. Explicit equations of all of them are in the ancillary file \texttt{results/realizable\_lattices.txt}.

\begin{table}[ht]
\centering\small
\begin{tabular}{lcccccl}
\toprule
& $(t_2,t_3,t_4,t_5,t_6)$ & $r_H$ & $m$ & field & realizations & freeness, remarks\\
\midrule
$L_0$ & $(0,25,0,3,0)$ & $6^{15}$ & 5 & $\QQ(\zeta_5)$ & 4 (one orbit) & free $(6,8)$; $\cA(5,5,3)$\\
$F_1$ & $(27,6,0,6,0)$ & $6^3\,7^{12}$ & 5 & $\QQ(a)$ & $\mathbb A^1$ minus 7 points & free $(7,7)$; Cuntz's family\\
$F_2$ & $(21,6,6,3,0)$ & $6^6\,7^9$ & 5 & $\QQ$ & 1 & free $(7,7)$\\
$F_3$ & $(21,6,6,3,0)$ & $6^6\,7^9$ & 5 & $\QQ(\sqrt{-3})$ & 2 (one orbit) & free $(7,7)$; $\subset\cA(G_{26})$\\
$F_4$ & $(21,6,6,3,0)$ & $6^6\,7^9$ & 5 & $\QQ(\sqrt5)$ & 2 (one orbit) & free $(7,7)$\\
$F_5$ & $(19,6,8,2,0)$ & $6^7\,7^8$ & 5 & $\QQ(\sqrt{-3})$ & 2 (one orbit) & free $(7,7)$; $\subset\cA(G_{26})$\\
$N_1$ & $(27,0,8,3,0)$ & $5\,6^2\,7^{12}$ & 5 & $\QQ(i)$ & 2 (one orbit) & not free\\
$N_2$ & $(22,8,4,2,1)$ & $4\,6^2\,7^{12}$ & 6 & $\QQ(i)$ & 2 (one orbit) & not free\\
$N_3$ & $(24,12,0,0,3)$ & $6^3\,7^{12}$ & 6 & $\QQ(\zeta_5)$ & 4 (one orbit) & not free; $\subset\cA(5,1,3)$\\
$N_4$ & $(17,16,0,1,2)$ & $6^6\,7^9$ & 6 & $\QQ(\zeta_5)$ & 4 (one orbit) & not free; $\subset\cA(5,1,3)$\\
$N_5$ & $(10,20,0,2,1)$ & $6^9\,7^6$ & 6 & $\QQ(\zeta_5)$ & 4 (one orbit) & not free; $\subset\cA(5,1,3)$\\
\bottomrule
\end{tabular}
\caption{The realizable lattices in $\cL(15;6,8)\cup\cL(15;7,7)$; $\cL(16;7,8)$ contains none. The column $r_H$ gives the distribution of the number of points per line ($6^3\,7^{12}$: three lines with $6$ points and twelve with $7$). ``Realizations'' counts labelled realizations modulo $\mathrm{PGL}_3(\CC)$.}
\label{tab:realizable}
\end{table}

\begin{proof}[Proof of the Main Theorem]
Let $L$ be a Terao counterexample with $d\in\{15,16\}$ lines. By Proposition~\ref{prop:reduction}, $L\in\cL(15;6,8)\cup\cL(15;7,7)\cup\cL(16;7,8)$. By Section~\ref{sec:enumeration}, these sets are completely known. By Section~\ref{sec:realization} and the table at the beginning of this section, exactly eleven of their elements are realizable over $\CC$. For each of them, freeness is constant on the set of all complex realizations:
\begin{itemize}
\item for the nine lattices $L_0$, $F_3$, $F_4$, $F_5$, $N_1,\dots,N_5$, the realizations form one Galois orbit, whose freeness is certified by Proposition~\ref{prop:cert};
\item for $F_2$, whose construction has cost $0$ and whose unique realization is defined over $\QQ$, by Lemma~\ref{lem:freecrit}: the matrix $M_{H,6}$ has trivial kernel over $\QQ$ (in the family criterion of Section~\ref{sec:realization}, $D$ is a nonzero constant);
\item for $F_1$, by the family criterion ($D\ne0$, and every irreducible factor of $D$ divides an element of $N$); see also the certificate in Section~\ref{subsec:ident}.
\end{itemize}
Hence $L$ cannot be a Terao counterexample.
\end{proof}

\begin{proof}[Proof of Theorem~\ref{thm:classification}]
Let $\cA$ be as in the theorem. Abe's restriction (proof of Lemma~\ref{lem:abe}) and the absence of divisional freeness give $r_H\le d_1$ for all $H$. Together with $m\le d_1-1$, $r_H\ge4$ and Lemmas~\ref{lem:counting} and~\ref{lem:boj}, $L(\cA)$ satisfies (R1)--(R5) for its type. Since $X>0$ we have $m\ge3$, hence $d_1\ge4$, so Table~\ref{tab:reduction} covers all possible types. Table~\ref{tab:reduction} leaves only the three types above, and $L(\cA)$ is one of the realizable lattices of Table~\ref{tab:realizable} that has a free realization. Conversely, these six lattices have $m=5\le d_1-1$, $r_H\in\{6,7\}$, and $r_H\notin\{d_1+1,d_2+1\}$ for all $H$; so all their complex realizations, which are free by Section~\ref{sec:results}, satisfy the hypotheses and are not divisionally free.
\end{proof}

\subsection{Identifications}\label{subsec:ident}
\begin{itemize}
\item \emph{$L_0$.} The reflection arrangement $\cA(5,5,3)$ is given by $\prod_{k=0}^4(x-\zeta^ky)(y-\zeta^kz)(x-\zeta^kz)$ with $\zeta=e^{2\pi i/5}$. Its $15$ lines form the $3$-net of the cyclic group of order $5$, and the five lines of each class are concurrent. There are exactly two main classes of Latin squares of order $5$. Correspondingly, $\cL(15;6,8)$ contains a second matroid with $(t_2,\dots,t_6)=(0,25,0,3,0)$: a $3$-net of the non-cyclic class whose three classes are pencils. It is not realizable, which is also seen directly. Each centre lies on the five lines of its class, and $m=5$, so no line of the arrangement passes through two centres.
\begin{itemize}
\item If the three centres are not collinear, take them to be the coordinate points. The lines are then $x=ay$, $y=bz$, $x=cz$ with $a\in A$, $b\in B$, $c\in C$, where $A,B,C\subset\CC^*$ have five elements, and concurrency means $c=ab$. So $aB=C$ for all $a\in A$, and $AA^{-1}$ lies in the stabilizer of $B$, a subgroup of order at most $5$. Hence $A$ and $B$ are cosets of $\mu_5$, and the Latin square is cyclic.
\item If the centres are collinear, a projective change of coordinates turns the classes into the parallel classes $x=a$, $y=b$, $x+y=c$. Then $A+B\subseteq C$, which contradicts $|A+B|\ge|A|+|B|-1=9$ in $(\CC,+)$.
\end{itemize}
The $3$-net alone, without concurrent classes, can be realized for a non-group Latin square of order $5$; compare \cite{Sti07}. Barakat and Kühne \cite[Sect.~8 of the arXiv version]{BK23} determined the non-free loci of the matroids underlying the reflection arrangements of the groups $G(n,n,3)$ for $n<10$; for prime $n$ they found that a representation is free if and only if the characteristic of the field is not $n$. For $n=5$ this already shows that every complex arrangement with lattice $L_0$ is free.
\item \emph{$N_3$, $N_4$, $N_5$.} The lattices of sub-arrangements below were computed exactly, with coefficients in the cyclotomic integers (ancillary script \path{identifications/ident.py}). The $816$ sub-arrangements of $15$ lines of $\cA(5,1,3)=\cA(5,5,3)\cup\{x,y,z\}$ have $11$ distinct lattices. Four of them satisfy (R1)--(R5): $L_0$ (remove $x,y,z$), and $N_3$, $N_4$, $N_5$, of type $(15;7,7)$.
\item \emph{$F_3$, $F_5$.} These are sub-arrangements of the reflection arrangement of the Shephard--Todd group $G_{26}$. With $\omega=e^{2\pi i/3}$, $\cA(G_{26})$ consists of the $21$ lines $x,y,z$, $x-\omega^jy$, $y-\omega^jz$, $x-\omega^jz$ and $x+\omega^jy+\omega^kz$ ($0\le j,k\le2$); it has $(t_2,t_3,t_4,t_5)=(36,0,9,12)$ and exponents $(1,7,13)$. Its $15$-line sub-arrangements have $206$ distinct lattices, and exactly two of them satisfy (R1)--(R5).
\item \emph{$\cA(6,1,3)$.} None of its $231$ lattices of $15$-line sub-arrangements satisfies (R1)--(R5).
\item \emph{$F_1$.} This is the lattice of Cuntz's arrangements $\cA_t$ \cite[§4.2]{Cun14}; see Appendix~\ref{app:blocks}. In Cuntz's parametrization, an exact computation (ancillary script \path{identifications/cuntz.py}) shows that $\cA_t$ has lattice $F_1$ exactly for
\[
t\notin\{0,1,\tfrac12\}\cup\{t:\ t^2+t-1=0\}\cup\{t:\ t^2-3t+1=0\}.
\]
\begin{itemize}
\item For $t\in\{0,1,\frac12\}$, some of the $15$ lines coincide.
\item At the four roots of $t^2+t-1$ and $t^2-3t+1$, the lattice becomes that of the Coxeter arrangement $\cA(H_3)$, with $(t_2,t_3,t_4,t_5)=(15,10,0,6)$ and exponents $(1,5,9)$.
\end{itemize}
In the published version \cite[§7]{ACKN16} the family is given in a different, symmetric parametrization, and \cite[Prop.~7.2]{ACKN16} describes its exceptional parameters correctly; its generic member also has lattice $F_1$ (checked by the same script). The description in the preprint \cite[§4.2]{Cun14} lists $-1$ and $\frac32\pm\sqrt2$ instead of $\frac{3\pm\sqrt5}2$ among the exceptional parameters. The arrangement $\cA_t$ is free with exponents $(1,7,7)$ \cite[Prop.~7.2]{ACKN16}, and Theorem~\ref{thm:classification} confirms that every complex arrangement with lattice $F_1$ is free.

The script \path{identifications/certificates.py} verifies the following short certificate. We use the normalized parameter $a$ of the ancillary files (not Cuntz's $t$); the arrangements $\cA(a)$ are given explicitly in \path{results/realizable_lattices.txt}. The irreducible factors of the polynomials in $N$ are exactly the factors of
\[
B(a)=a(a-1)(a+1)(a^2-a-1)(a^2+a-1),
\]
so by Proposition~\ref{prop:onepar}(c) the realizations of $F_1$ are the arrangements $\cA(a)$ with $B(a)\ne0$. The first line of $\cA(a)$ is $z=0$. In the coordinates $(x,y)$ on it, the Ziegler multirestriction is
\[
x^2(x+y)^4y^4\,(x-ay)\,\bigl((a-1)x-y\bigr)\,\bigl((1-a)x-ay\bigr)\,(ax-y).
\]
Represent its points by $(s_P,t_P)=(0,1),(1,-1),(1,0),(a,1),(1,a-1),(a,1-a),(1,a)$, in this order, and use the exact system of Section~\ref{subsec:ziegler}, with Hasse derivatives in $X$ at the points with $t_P\ne0$ and in $Y$ at $(1,0)$. With the rows ordered by points and then by increasing order of the derivative, and the unknowns ordered as the coefficients of $g_1$ and then of $g_2$, by decreasing powers of $X$, the $14\times14$ matrix of conditions for a derivation of degree $6$ satisfies
\[
\det M(a)=2\,a^6(a-1)^4(a+1)^4(a^2-a-1)^2(a^2+a-1)^2.
\]
So $\det M(a)\ne0$ whenever $B(a)\ne0$, and by Lemma~\ref{lem:freecrit} every realization of $F_1$ is free.
\end{itemize}
We have not tried to identify $F_2$, $F_4$, $N_1$ and $N_2$ in the literature. $F_2$ is a real arrangement defined over $\QQ$, with defining polynomial
\begin{multline*}
xyz\,(x+y)(x+z)(y+2z)(x+y+z)(x+y+2z)(x-y-z)\\
\cdot(x-y+2z)(x+2y+2z)(2x+y+z)(2x+y+2z)(2x+y+4z)(2x+y-2z).
\end{multline*}

\section{Reproducibility}\label{sec:reproducibility}

The ancillary archive contains:
\begin{itemize}
\item the C programs \texttt{heavyaug} (Stage~1) and \texttt{lsgen3} (Stage~2), which require nauty;
\item the Python modules for the exact analysis of Sections~\ref{sec:realization} and~\ref{sec:results}, which use python-flint \cite{FLINT} and SymPy; the class of cost~$2$ is excluded by the elementary certificate of Section~\ref{subsec:several};
\item the scripts that compute the identifications of Section~\ref{subsec:ident} exactly and verify the short certificates of Sections~\ref{subsec:several} and~\ref{subsec:ident}, in the folder \texttt{identifications};
\item the results, including explicit equations of all realizable lattices;
\item a script \texttt{run.sh} that reruns the enumeration and screening of each type, checks the expected counts and canonical-list checksums, and compares all verdicts and detailed records with the supplied results. Its \texttt{--certificates} mode recomputes the eleven realizable cases, the cost-$2$ obstruction and the stated identifications, without rerunning the exhaustive enumeration.
\end{itemize}
Every entry of the lists is checked to encode a linear space satisfying (R1)--(R5) when it is screened, and the final verdict rejects any missing, duplicate, out-of-range or inconsistent record. The original computation used gcc 13.3, nauty 2.8.8, \textsc{Singular} 4.3.2, Python 3.13 and python-flint 0.9.0. Exact screening took about $0.9$--$1.2$\,ms per class: $1.5$ CPU-hours for $(15;7,7)$ and $5.5$ for $(16;7,8)$. The lists $\cL(d;d_1,d_2)$, as sorted files of canonical \texttt{graph6} strings, have SHA-256 checksums
\begin{center}\footnotesize
\begin{tabular}{ll}
$\cL(15;6,8)$ & \texttt{8de6cde5fc270e15cae212de4ee41d5202197e8aeab78cf6e65614650746dcd9}\\
$\cL(15;7,7)$ & \texttt{f1be0ac91ef3d5e230d2812b08c3e67e2090445398d9bfd5f7905ab0677a9040}\\
$\cL(16;7,8)$ & \texttt{12a4cb695442e91e9d7fc55e9d5abc608c8cf99c14d96f308a4876e668fc9f4d}
\end{tabular}
\end{center}
These refer to the canonical forms of nauty 2.8.8; the reproduction script requires these checksums and stops on a mismatch. Generated lists and intermediate output are not distributed, since the supplied programs regenerate them. The supplied reproduction scripts do not require Singular.

\section*{Statements and declarations}

\noindent\textbf{Competing interests.} The authors have no financial interests to disclose.

\noindent\textbf{Data availability.} The code and the results of the computations are available as ancillary files of the arXiv version of this paper (Section~\ref{sec:reproducibility}).

\noindent\textbf{Use of generative AI.}
Claude (Anthropic) was used for literature search, programming, and drafting the text; ChatGPT (OpenAI) for auditing, and editorial assistance. After using these tools, the authors reviewed and edited the content and takes full responsibility for the content of the publication.

\appendix
\section{Blocks of the realizable lattices}\label{app:blocks}
Lines are numbered $1,\dots,9,\mathtt a,\dots,\mathtt f$ (so $\mathtt a=10,\dots,\mathtt f=15$). Each block, a point of multiplicity at least $3$, is written as the word of its lines. The numbering is the canonical one of the ancillary file.

\medskip
\begingroup\raggedright
\noindent\textbf{$L_0$}, type $(15;6,8)$, $\cA(5,5,3)$.\enspace{\small\texttt{12345 6cdef 789ab 167 18c 19d 1af 1be 269 27c 28e 2ad 2bf 368 37d 39f 3ae 3bc 46b 47f 48d 49e 4ac 56a 57e 58f 59c 5bd}}\par\smallskip
\noindent\textbf{$F_1$}, type $(15;7,7)$, Cuntz's family.\enspace{\small\texttt{13bcf 149ae 2367e 2458f 57acd 689bd 12d 34d 5be 6af 79f 8ce}}\par\smallskip
\noindent\textbf{$F_2$}, type $(15;7,7)$.\enspace{\small\texttt{124ad 136be 235cf 189c 279b 378a 457e 469f 568d 7df 8ef 9de abf ace bcd}}\par\smallskip
\noindent\textbf{$F_3$}, type $(15;7,7)$, sub-arrangement of $\cA(G_{26})$.\enspace{\small\texttt{127ad 138be 239cf 145c 246b 356a 489d 579e 678f 4ef 5df 6de 7bc 8ac 9ab}}\par\smallskip
\noindent\textbf{$F_4$}, type $(15;7,7)$.\enspace{\small\texttt{145ab 278ac 369bc 168e 179d 249e 256f 347f 358d 46d 57e 89f adf bef cde}}\par\smallskip
\noindent\textbf{$F_5$}, type $(15;7,7)$, sub-arrangement of $\cA(G_{26})$.\enspace{\small\texttt{1389b 247ab 126c 145f 235e 346d 578c 59ad 679f 68ae 7de 8df 9ce acf bcd bef}}\par\smallskip
\noindent\textbf{$N_1$}, type $(15;7,7)$.\enspace{\small\texttt{12345 169cd 178ab 2ade 2bcf 37ce 38df 46af 489e 56be 579f}}\par\smallskip
\noindent\textbf{$N_2$}, type $(15;7,7)$.\enspace{\small\texttt{123456 18def 19abc 279d 37be 47cf 578a 2ae 39f 48b 5cd 689 6af 6bd 6ce}}\par\smallskip
\noindent\textbf{$N_3$}, type $(15;7,7)$, sub-arrangement of $\cA(5,1,3)$.\enspace{\small\texttt{124567 138ace 239bdf 4ab 4cf 4de 589 5ad 5bc 68b 69a 6ef 78f 79e 7cd}}\par\smallskip
\noindent\textbf{$N_4$}, type $(15;7,7)$, sub-arrangement of $\cA(5,1,3)$.\enspace{\small\texttt{5bcdef 6789af 12345 16b 17e 18d 1ac 27c 28e 29d 2ab 36c 37d 39b 3ae 46e 48b 49c 4ad}}\par\smallskip
\noindent\textbf{$N_5$}, type $(15;7,7)$, sub-arrangement of $\cA(5,1,3)$.\enspace{\small\texttt{123456 1789a bcdef 27b 28d 29f 2ae 37e 38c 39d 3af 47f 48b 49c 4ad 57c 58f 59e 5ab 67d 68e 69b 6ac}}\par\smallskip
\endgroup

\medskip
\noindent\emph{Cuntz's family.} Following \cite[§4.2]{Cun14}, $\cA_t$ is defined by the normal vectors
\[
\begin{gathered}
(1,0,0),\ (1,1,0),\ (1,0,1),\ (1,1,1),\ (1,t,1),\ (0,1,0),\ (2,1,1),\ (t+1,t,1),\\
(t+1,1,1),\ (2t,t,1),\ (1,1-t,1),\ (1-3t,\ t^2-3t+1,\ -t),\\
(3t-1,\ t,\ t),\ (1-3t,\ -t^2,\ -t),\ (3t-1,\ 2t-1,\ t).
\end{gathered}
\]

\end{document}